\documentclass[review]{elsarticle}

\usepackage{xcolor}
\usepackage[a4paper, left=3.5cm, right=3.5cm, top=2.5cm, bottom=2.5cm]{geometry} 

\usepackage{booktabs} % for toprule, bottomrule

\usepackage{enumerate}

\usepackage[utf8]{inputenc}
\usepackage{amssymb,amsmath,amsfonts,accents,amsthm,thmtools}
\usepackage{mathtools}
\usepackage[colorlinks=true,citecolor=blue,linkcolor=blue,urlcolor=blue]{hyperref}

\usepackage[capitalise]{cleveref}
\crefname{thm}{Theorem}{Theorems}
\usepackage{bm}

\usepackage{times}
\usepackage[english]{babel}

\makeatletter
\def\bbl@set@language#1{%
	\edef\languagename{%
		\ifnum\escapechar=\expandafter`\string#1\@empty
		\else\string#1\@empty\fi}%
	\@ifundefined{babel@language@alias@\languagename}{}{%
		\edef\languagename{\@nameuse{babel@language@alias@\languagename}}%
	}%
	\select@language{\languagename}%
	\expandafter\ifx\csname date\languagename\endcsname\relax\else
	\if@filesw
	\protected@write\@auxout{}{\string\select@language{\languagename}}%
	\bbl@for\bbl@tempa\BabelContentsFiles{%
		\addtocontents{\bbl@tempa}{\xstring\select@language{\languagename}}}%
	\bbl@usehooks{write}{}%
	\fi
	\fi}
\newcommand{\DeclareLanguageAlias}[2]{%
	\global\@namedef{babel@language@alias@#1}{#2}%
}
\makeatother

\DeclareLanguageAlias{en}{english}
\DeclareLanguageAlias{English}{english}
\DeclareLanguageAlias{Englisch}{english}
\DeclareLanguageAlias{EN}{english}
\DeclareLanguageAlias{en-US}{english}
\usepackage{graphicx}
\makeatletter
\def\th@plain{%
	\thm@notefont{}% same as heading font
	\itshape % body font
}
\def\th@definition{%
	\thm@notefont{}% same as heading font
	\normalfont % body font
}
\makeatother

\newtheorem{theorem}{Theorem}

\theoremstyle{plain}

\newcommand{\ham}{H}

\newcommand{\evecU}{\bm{\Phi}}
\newcommand{\evecUnPadded}{\bm{\phi}}

\newcommand{\sbar}{\overline{S}}

\newcommand{\ISR}[2]{R_{#1}(#2,\lambda)}
\newcommand{\sbb}{\sbar{}\sbar}
\newcommand{\ssb}{S\sbar}
\newcommand{\sbs}{\sbar{}S}
\newcommand{\symm}{T}
\newcommand{\symmEval}{t}
\newcommand{\unitVec}[1]{\bm{e}_{#1}}
\begin{document}
\begin{frontmatter}
\title{On symmetries of a matrix and its isospectral reduction}

\author[1]{M. Röntgen\corref{cor1}%
	\fnref{fn1}}
\ead{mroentge@physnet.uni-hamburg.de}

\author[1]{M. Pyzh\fnref{fn1}}%

\author[1]{C. V. Morfonios}%

\author[1,2]{P. Schmelcher}
\address[1]{%
	Zentrum für Optische Quantentechnologien, Fachbereich Physik, Universität Hamburg, Luruper Chaussee 149, 22761 Hamburg, Germany
}%
\address[2]{%
	The Hamburg Centre for Ultrafast Imaging, Universität Hamburg, Luruper Chaussee 149, 22761 Hamburg, Germany
}%
\cortext[cor1]{Corresponding author}
\fntext[fn1]{These two authors contributed equally.}
\begin{abstract}
The analysis of diagonalizable matrices in terms of their so-called isospectral reduction represents a versatile approach to the underlying eigenvalue problem.
Starting from a symmetry of the isospectral reduction, we show in the present work that it is possible to construct a corresponding symmetry of the original matrix.
\end{abstract}

\end{frontmatter}

\section{Introduction}
\label{sec:intro}

The study of matrix eigenvalue problems of the form $\ham \mathbf{x} = \lambda \mathbf{x}$ is ubiquitous in science and technology.
A promising direction in analysing this eigenvalue problem is a dimensional reduction of $\ham$.
In this work, we consider the so-called \emph{isospectral reduction} (ISR) \cite{Bunimovich2014IsospectralTransformationsNewApproach}, which is defined via matrix partitioning of $\ham$ as
\begin{equation} \label{eq:ISR}
	\ISR{S}{\ham} = \ham_{SS} + \ham_{\ssb} \left(\lambda - \ham_{\sbb} \right)^{-1} \ham_{\sbs}\,,
\end{equation}
where the set $S \subseteq \{1,\ldots{},N\}$ and its complement $\sbar$ are used for partitioning $\ham \in \mathbb{C}^{N\times N}$.
For example, $\ham_{\ssb}$ denotes the submatrix obtained from $\ham$ by taking the rows in $S$ and the columns in $\sbar$.
The ISR provides valuable insights in quantum physics, where it is referred to as an effective Hamiltonian obtained from subsystem partitioning \cite{Grosso2013SolidStatePhysics,Rontgen2021PRL126180601LatentSymmetryInducedDegeneracies}.

As its name suggests, the ISR preserves the spectral properties of $\ham$:
Defining the multiset $\sigma(M)$ as the eigenvalue spectrum of a matrix $M$, it has been shown that the non-linear eigenvalue spectrum of $\ISR{S}{\ham}$ fulfills $\sigma(\ISR{S}{\ham}) = \sigma(\ham) - \sigma(\ham_{\sbb})$ \cite{Bunimovich2014IsospectralTransformationsNewApproach}.
Thus, whenever $\ham$ and $\ham_{\sbb}$ share no eigenvalues, $\ISR{S}{\ham}$ preserves the eigenvalue spectrum of $\ham$.
Building on this favourable property, the ISR has been applied \cite{Bunimovich2012LAIA4371429IsospectralGraphReductionsImproved} to improve the eigenvalue approximations of Gershgorin, Brauer, and Brualdi \cite{Gerschgorin1931IANSSM7749UberAbgrenzungEigenwerteMatrix,Brauer1947DMJ1421LimitsCharacteristicRootsMatrix,Brualdi1982LMA11143MatricesEigenvaluesDirectedGraphs}, to study pseudo-spectra of graphs and matrices\cite{VasquezFernandoGuevara2014NLAA22145PseudospectraIsospectrallyReducedMatrices}, to create stability preserving transformations of networks \cite{Bunimovich2011N25211IsospectralGraphTransformationsSpectral,Bunimovich2013N262131RestrictionsStabilityTimedelayedDynamical,Reber2020N332660IntrinsicStabilityStabilityDynamical}, to study the survival probabilities in open dynamical systems \cite{Bunimovich2014ETODCS119ImprovedEstimatesSurvivalProbabilities}, and very recently also to explain spectral degeneracies of physical systems \cite{Rontgen2021PRL126180601LatentSymmetryInducedDegeneracies}.

In this work, we concentrate on symmetries of the isospectral reduction, which we define as a normal and invertible matrix $\symm$ which commutes with $\ISR{S}{\ham}$, that is, $\left[ \ISR{S}{\ham}, \symm\right] = 0$, for all $\lambda \notin \sigma(\ham_{\sbb})$.
We note that, for the special case of permutations, such symmetries of $\ISR{S}{\ham}$ have been coined ``latent'' or ``hidden'' symmetries of $\ham$ \cite{Smith2019PA514855HiddenSymmetriesRealTheoretical}.
In that context, $\ham$ is the (weighted) adjacency matrix of a graph, and its automorphisms are described by a permutation matrix commuting with $\ham$.
The term ``latent'' then refers to the fact that the ISR of $\ham$ may feature non-trivial permutation symmetries, while $\ham$ features only a trivial (namely: the identity operation) permutation symmetry.
Interestingly, latent symmetries have been recently connected \cite{Kempton2020LAIA594226CharacterizingCospectralVerticesIsospectral} to the theory of so-called ``cospectral vertices'' which find applications in quantum computing \cite{Godsil2012PRL109050502NumberTheoreticNatureCommunicationQuantum,Kay2018AQPerfectStateTransferGraph,Rontgen2020PRA101042304DesigningPrettyGoodState}.
In the remainder of this work, we will adapt the generalized notion of Ref. \cite{Rontgen2021PRL126180601LatentSymmetryInducedDegeneracies}, and denote also non-permutation symmetries of $\ISR{S}{\ham}$ as ``latent symmetries of $\ham$''.

The nonlinear eigenvalues of $\ISR{S}{\ham}$ correspond to eigenvectors fulfilling $\ISR{S}{\ham} \mathbf{y} = \lambda \mathbf{y}$, given by the projection $\mathbf{y} = \mathbf{x}_{S}$ of the eigenvector $\bm{x}$ of $\ham$ to $S$ \cite{Duarte2015LAIA474110EigenvectorsIsospectralGraphTransformations}.
This allows to derive the profound impact of latent symmetries on the eigenvectors of $\ham$: Whenever the symmetry $\symm$ has only simple eigenvalues, and additionally $\ham$ and $\ham_{\sbb}$ share no eigenvalues, then all eigenvectors of $\ham$ fulfill
\begin{equation} \label{eq:localSymmetry}
\symm \mathbf{x}_{S} = \symmEval \mathbf{x}_{S} \, ,
\end{equation}
with $\symmEval$ being an eigenvalue of $\symm$.
However, when $\ham$ and $\ham_{\sbb}$ share eigenvalues, the behavior of the corresponding eigenvectors is still an open issue for general $\symm$.
Recently, an interesting first step in solving this problem has been made for the special case of $\symm = P = \begin{psmallmatrix}
	0 & 1 \\
	1 & 0
\end{psmallmatrix}$ \cite{Kempton2020LAIA594226CharacterizingCospectralVerticesIsospectral,Godsil2017A1StronglyCospectralVertices}:
It has been shown that $\left[\ISR{S}{\ham}, P  \right] = 0$ corresponds to the existence of an orthogonal block-diagonal matrix $Q = P \oplus \overline{Q}$ fulfilling $\left[Q,\ham \right] = 0$.
Thus, $\ham$ and $Q$ can be simultaneously diagonalized, and---assuming no degeneracies of $\ham$---it follows that \emph{all} eigenvectors either fulfill \cref{eq:localSymmetry} or vanish on $S$, that is, $\mathbf{x}_{S} = 0$.
We generalize this result in the following.
\begin{theorem}
	Let $\ISR{S}{\ham}$ be the isospectral reduction of a self-adjoint matrix
	\begin{equation} \label{eq:HPart}
	\ham = \begin{pmatrix}
	\ham_{SS} &  \ham_{\ssb} \\
	\ham_{\sbs} & \ham_{\sbb}
	\end{pmatrix}
	\end{equation}
	and $\symm$ be a normal and invertible $|S|\times |S|$ matrix.
	Then the following are equivalent
	\begin{enumerate}[(i)]
		\item $\left[\mathcal{R}_{S}(\ham,\lambda),\symm \right] = 0 \;\forall\; \lambda \notin \sigma(\ham_{\sbb})$.
		\item $\left[ \left(\ham^{k} \right)_{SS},\symm \right] = 0 \;\forall\; k \in \mathbb{N}$.
		\item There exists a normal matrix $Q = \symm \oplus \overline{Q}$ fulfilling $\left[Q,\ham \right] = 0$.
	\end{enumerate}
\end{theorem}
\begin{proof}
	The equivalence of (i) and (ii) has already been proven in Ref.  \cite{Rontgen2021PRL126180601LatentSymmetryInducedDegeneracies}.
	Proving $(iii) \Rightarrow (ii)$ is trivial, since
	\begin{equation}
		\left[\ham,Q \right] = 0 \;\Rightarrow\; \left[\ham^{k},Q \right] = 0,
	\end{equation}
	and writing the $SS$-block of this commutator gives $\left[\left(\ham^{k} \right)_{SS},Q_{SS} \right] = \left[\left(\ham^{k} \right)_{SS},\symm \right] = 0$.
	
	We now prove the remaining step of $(ii) \Rightarrow (iii)$:
Since $\symm$ is normal and invertible, it can be spectrally decomposed as $\symm= \sum_{i=1}^{n} \sum_{j=1}^{d_i} \symmEval_{i} \evecUnPadded_{i,j} \evecUnPadded_{i,j}^{\dagger}$ with each of its nonzero eigenvalues $\symmEval_{i}$ corresponding to a set of $d_i$ orthonormal eigenvectors $\evecUnPadded_{i,j}$ with degeneracy index $j=1,\ldots{},d_{i}$, and with $\dagger$ denoting the conjugate transpose.

Let now $\evecU_{i,j}$ be the $N$-dimensional vector obtained from 
$\evecUnPadded_{i,j}$ by padding it with zeros such that $\left(\evecU_{i,j}\right)_{S} = \evecUnPadded_{i,j}$ and $\left(\evecU_{i,j}\right)_{\sbar} = 0$, with $N$ denoting the dimension of $\ham$.
We denote by $K_{i,j} = span\left(\evecU_{i,j},\ham \evecU_{i,j},\ldots{}, \ham^{N-1} \evecU_{i,j} \right)$ the Krylov subspace generated by $\evecU_{i,j}$.
As we now prove, when $l\ne m$, $K_{l,j} \perp K_{m,j'}$ for all $j,j'$.
Equivalently, the hermitian inner product $\langle \ham^{k_{1}}\evecU_{l,j}, \ham^{k_{2}} \evecU_{m,j'}\rangle=0$ for all $k_{1},k_{2}$:
Since $\ham = \ham^{\dagger}$ is self-adjoint, this inner product can be written as
\begin{align}
	\evecU_{l,j}^{\dagger} \, \ham^{k}\, \evecU_{m,j'} \overset{(a)}{=}& \evecUnPadded_{l,j}^{\dagger} \, \left(\ham^{k}\right)_{SS} \evecUnPadded_{m,j'} = \evecUnPadded_{l,j}^{\dagger} \, \symm^{-1} \left(\ham^{k}\right)_{SS} \symm\, \evecUnPadded_{m,j'} \\
	=&\frac{\symmEval_{m}}{\symmEval_{l}} \, \evecUnPadded_{l,j}^{\dagger} \left(\ham^{k} \right)_{SS} \evecUnPadded_{m,j'} \overset{(b)}{=} 0
\end{align}
with $k = k_{1} + k_{2}$ and where we have used that (a) both $\evecU_{l,j}$ and $\evecU_{m,j'}$ vanish on $\sbar$ and (b) that $\symmEval_{l} \ne \symmEval_{m}$.

We then proceed by defining the $\ham$-invariant subspaces $\widetilde{K}_{i} = \bigoplus_{j} K_{i,j}$.
From the above, it is clear that $\widetilde{K}_{l} \perp \widetilde{K}_{m}$ when $l\ne m$.
We now construct an orthonormal basis of each $\widetilde{K}_{i}$ as follows: As the first $d_{i}$ basis vectors, we choose the generating vectors $\evecU_{i,1},\ldots{},\evecU_{i,d_{i}}$ of the Krylov spaces $K_{i,1},\ldots{},K_{i,d_{i}}$.
These vectors are already pairwise orthonormal and are necessarily contained in $\widetilde{K}_{i}$.
Denoting the dimension of $\widetilde{K}_{i}$ by $\tilde{d}_{i}$, the remaining $\tilde{d}_{i} - d_{i} = r_{i} \ge 0$ basis vectors $\overline{\evecU}_{i,1},\ldots{},\overline{\evecU}_{i,r_{i}}$ can be shown to vanish on $S$: 
First, being a basis vector of $\widetilde{K}_{i}$, each $\overline{\evecU}_{i,j}$ must be orthogonal to all other basis vectors of this space, and in particular to $\evecU_{i,j'}$ for all $j'$.
Second, since $\widetilde{K}_{i} \perp \widetilde{K}_{i'}$ with $i \ne i'$, each $\overline{\evecU}_{i,j}$ must be orthogonal to each basis vector of $\widetilde{K}_{i'}$, and in particular to $\evecU_{i',j'}$ for all $j'$.
Thus, $\overline{\evecU}_{i,j}$ is orthogonal to $\evecU_{i',j'}$ for all $i',j'$.
Now, since the set $\{\evecU_{i',j'} \}$ forms an orthogonal basis for any vector that vanishes on $\sbar$, and since each element of this set vanishes on $\sbar$, it follows that each $\overline{\evecU}_{i,j}$ must vanish on $S$.

The above insights allow us to finally construct
\begin{equation} \label{eq:DefinitionOfQ}
Q = \sum_{i=1}^n \symmEval_{i} \Big[\sum_{j=1}^{d_{i}} \evecU_{i,j} \evecU_{i,j}^{\dagger} + \sum_{j=1}^{r_{i}} \overline{\evecU}_{i,j} \overline{\evecU}_{i,j}^{\dagger} \Big] \, .
\end{equation}
By construction, $Q$ is a normal matrix.
We now prove that $\left[\ham,Q\right] = 0$.
To this end, let $\mathbf{x}_{i} \in \widetilde{K}_{i}$.
It follows that $\ham \mathbf{x}_{i} \in \widetilde{K}_{i}$ as well, since $\widetilde{K}_{i}$ is by construction an $\ham$-invariant subspace.
Then, by \cref{eq:DefinitionOfQ}, all vectors in $\widetilde{K}_{i}$ are eigenvectors of $Q$ with identical eigenvalue $\symmEval_{i}$.
We thus have $Q \ham \mathbf{x}_{i} = \symmEval_{i} \ham \mathbf{x}_{i}$
and also $\ham Q \mathbf{x}_{i} = \ham \left(\symmEval_{i} \mathbf{x}_{i} \right)$ .
Let now $V$ denote the orthogonal complement of $\bigoplus_{i} \widetilde{K}_{i}$.
It is obvious that for $\mathbf{v}\in V$ we have $Q \mathbf{v} = 0$ and thus also $H Q \mathbf{v} = 0$.
Being the orthogonal complement of $\ham$-invariant subspaces, $V$ is also $\ham$-invariant, and we also get $\ham \mathbf{v} \in V$ implying
$Q \ham \mathbf{v} = 0$.
In summary,
\begin{equation}
	Q H \mathbf{x} = H Q \mathbf{x}
\end{equation}
\emph{for any} $\mathbf{x}$.
Thus, $\left[Q,\ham \right] = 0$ as claimed.

We proceed by showing that $Q_{SS} = \symm$.
To this end, we define the vector $\unitVec{i}$ as having a $1$ on component $i$ and with all other components vanishing.
Then, for $s,s' \in S$, the matrix element $Q_{s,s'} = \unitVec{s}^{\dagger} Q \unitVec{s'}$, and since each basis vector $\overline{\evecU}_{i,j}$ vanishes on $S$, we have
\begin{equation}
	Q_{s,s'} = \sum_{i=1}^{n} \sum_{j=1}^{d_{i}} \symmEval_{i} \unitVec{s}^{\dagger} \evecU_{i,j} \evecU_{i,j}^{\dagger} \, \unitVec{s'} = \sum_{i=1}^{n} \sum_{j=1}^{d_{i}} \symmEval_{i} \left( \evecU_{i,j} \evecU_{i,j}^{\dagger} \right)_{s,s'} = \sum_{i=1}^{n} \sum_{j=1}^{d_{i}} \symmEval_{i} \left( \evecUnPadded_{i,j} \evecUnPadded_{i,j}^{\dagger} \right)_{s,s'} = \symm_{s,s'}
\end{equation}
where in the last step we recognized the spectral decomposition of $\symm$.

To see that $Q_{\ssb} = 0$ and also $Q_{\sbs} = 0$, it suffices to note that, due to \cref{eq:DefinitionOfQ}, $\unitVec{s}^{\dagger} Q \unitVec{\bar{s}} = 0$ for any $s \in S$, $\bar{s} \in \sbar$.
\end{proof}

%\bibliography{/afs/physnet.uni-hamburg.de/users/zoq_t/mroentge/LiteratureDB/Bibtex.bib}

\end{document}